\documentclass[11pt]{amsart}

\usepackage[english]{babel}     
\usepackage[utf8]{inputenc}
\usepackage{amsmath}
\usepackage{amssymb}
\usepackage{mathrsfs}
\usepackage{stmaryrd}
\usepackage{hyperref}
\usepackage{xcolor}
\usepackage{todonotes}
\usepackage{comment}
\usepackage{enumerate}
\usepackage{tikz-cd}
\usepackage[capitalise]{cleveref}

\usepackage{tikz}
\usetikzlibrary[topaths]

\newtheorem{theorem}{Theorem}[section]
\newtheorem{theoremx}{Theorem}

\newtheorem{lemma}[theorem]{Lemma}
\newtheorem{proposition}[theorem]{Proposition}

\theoremstyle{definition}
\newtheorem{definition}[theorem]{Definition}

\newtheorem{remark}[theorem]{Remark}

\theoremstyle{remark}
\numberwithin{equation}{section}

\let\epsilon\varepsilon

\newcommand{\alg}		{{\operatorname{\mathrm{alg}}}}

\newcommand{\cL}{\mathcal{L}}

\newcommand{\hh}{\mathbf{h}}

\newcommand{\uu}{\mathbf{u}}
\newcommand{\vv}{\mathbf{v}}
\newcommand{\ww}{\mathbf{w}}

\renewcommand{\gg}{\mathbf{g}}

\DeclareMathOperator{\Link}{Link}
\DeclareMathOperator{\Star}{Star}

\DeclareMathOperator{\Tr}{Tr}

\DeclareMathOperator{\Span}{Span}
\DeclareMathOperator{\Mat}{Mat}

\DeclareMathOperator{\Nor}{Nor}
\DeclareMathOperator{\qNor}{qNor}

\newcommand{\CC}{\mathbb{C}}

\newcommand{\EE}{\mathbb{E}}
\newcommand{\FF}{\mathbb{F}}

\newcommand{\ZZ}{\mathbb{Z}}

\newcommand{\calL}{\mathcal{L}}

\usepackage[date=year,backend=bibtex,doi=false,isbn=false,url=false,eprint=false]{biblatex}
\title[Right-angled Coxeter groups with a strongly solid von Neumann algebra]{Classification of right-angled Coxeter groups  with a strongly solid von Neumann algebra}

\date{\noindent \today.  MSC2010: 46L10,  46L51. Key words: strong solidity, von Neumann algebras,  right-angled Coxeter groups.   MC is  supported by the NWO Vidi grant VI.Vidi.192.018 `Non-commutative harmonic analysis and rigidity of operator algebras'. }

\author{Matthijs Borst and Martijn Caspers}

\address{TU Delft, EWI/DIAM,
	P.O.Box 5031,
	2600 GA Delft,
	The Netherlands}

\email{M.J.Borst@tudelft.nl}
\email{M.P.T.Caspers@tudelft.nl}

\begin{document}

\maketitle

\begin{abstract} Let $W$ be a finitely generated right-angled Coxeter group with group von Neumann algebra $\mathcal{L}(W)$. We prove the following dichotomy: either $\mathcal{L}(W)$ is strongly solid or $W$ contains $\mathbb{Z} \times \mathbb{F}_2$ as a subgroup.
This proves in particular strong solidity of $\mathcal{L}(W)$ for all non-hyperbolic Coxeter groups that do not contain $\mathbb{Z} \times \mathbb{F}_2$.
 
\end{abstract}

In their seminal paper \cite{OzawaPopaAnnals}  Ozawa and Popa gave a new proof of the fact that the free group factors $\mathcal{L}(\mathbb{F}_n)$ do not possess a so-called Cartan subalgebra. This result was obtained earlier by Voiculescu \cite{VoiculescuGAFA} using his free entropy. In fact Ozawa and Popa proved a stronger result: they showed that whenever $A$ is a diffuse amenable von Neumann subalgebra of $M := \mathcal{L}(\mathbb{F}_n)$ then the normaliser of $A$ in $M$ (see preliminaries) generates a von Neumann algebra that is still amenable. The latter property was then called `strong solidity' after Ozawa's notion of solid von Neumann algebras \cite{OzawaActa}.  Nowadays there are many examples of strongly solid von Neumann algebras. In particular in  \cite{ChifanSinclairENS}, and afterwards also \cite{PopaVaesCrelle}, it was proved that the group von Neumann algebra of any hyperbolic icc group is strongly solid.

A (finitely generated) right-angled Coxeter group $W$  with finite generating set $\Gamma$  subject to the relations that any two generators are either free, or they commute. This can also be described as follows. Let $\Gamma$ be a simple graph (finite, undirected, no double edges, no self-loops) and also write $\Gamma$ for its vertex set. Then $W$ is the finitely presented group with generating set given by the vertices of $\Gamma$ subject to the relations that $u^2 = e$ for all $u \in \Gamma$ and  $u,v \in \Gamma$ commute if they share an edge, and otherwise they are free. In terms of graph products \cite{greenGraphProductsGroups1990} this means that $W = \ast_{\Gamma} \mathbb{Z}_2$ where $\mathbb{Z}_2$ is the group with 2 elements and $\ast_\Gamma$ denotes the graph product over $\Gamma$, with all vertex groups equal to $\mathbb{Z}_2$.

It is natural to ask which properties of a graph and its vertex groups are reflected in the graph product von Neumann algebra. In the extreme case that $\Gamma$ has no edges this leads to the free factor problem \cite{DykemaDuke}, \cite{RadulescuInventiones}, \cite[Paragraph before Definition 2.2]{CaspersSkalskiWasilewski} and one may wonder how far a general Coxeter group is away from this situation. More precisely, our aim is to understand how much of the graph $\Gamma$, and thus commutation relations, can be recovered from $\mathcal{L}(W)$. Such rigidity questions are relevant for general graph products of operator algebras \cite{CaspersFima} and recent rigidity results   have been obtained in \cite{CaspersAPDE}, \cite{BorstCaspersWasilewski}, \cite{ChifanSri}, \cite{CDD1}, \cite{CDD2}. In particular in  \cite{CDD1}, \cite{CDD2} rather strong  rigidity results are obtained   in case the graph $\Gamma$ as a particular flower-type shape and the vertex groups are certain property (T) generalized wreath products. This paper concerns the other extreme case: the vertex groups are as small as possible and we try to recover some of the structure of $\Gamma$.

\vspace{0.3cm}

The main result of this paper completes the classification of right-angled Coxeter groups $W$ for which $\mathcal{L}(W)$ is strongly solid. This result completely clarifies \cite[Remark 5.6]{CaspersAPDE}. Namely we prove the following dichotomy.

\begin{theoremx}
    Let $W$ be a finitely generated right-angled Coxeter group. Then one of the following two holds:
    \begin{enumerate}
      \item\label{Item=Dicho1} $W$ contains $\mathbb{Z} \times \mathbb{F}_2$ as a subgroup.
      \item $\cL(W)$ is strongly solid.
    \end{enumerate}
\end{theoremx}

In fact our main theorem is  stronger, see Theorem \ref{Thm=Dichotomy}, which also shows that $\mathbb{Z} \times \mathbb{F}_2$  must in situation \eqref{Item=Dicho1} be located in a special subgroup. It easily follows that if $W$ contains $\mathbb{Z} \times \mathbb{F}_2$  then $\mathcal{L}(W)$ cannot be strongly solid. The difficult part is to show that this is the only obstruction. Our strong solidity result  is new in case $W$ is not hyperbolic and not equal to some free product of amenable Coxeter subgroups. There are many such Coxeter groups. Indeed, $\mathbb{Z}_2 \ast \mathbb{Z}_2 = D_\infty$ where $D_\infty$ is the infinite dihedral group.  Then $D_\infty \times D_\infty$ is not hyperbolic and nor is any $W$ that contains $D_\infty \times D_\infty$.

Our proof is based on rigidity of amalgamated free products and uses some of the main results of Ioana \cite{IoanaAnnENS} and Vaes \cite{VaesPrims}; in particular we use the three alternatives of \cite[Theorem A]{VaesPrims}. We show that these results can be exploited in a clean and elegant way to get  stronger results  by combining multiple amalgamated free product decompositions at the same time. This requires to have sufficient control over embeddings of normalizers in graph products and relative amenability with respect to various subgroups; we obtain such results in Sections \ref{Sect=QuasiNormalizer} and \ref{Sect=RelativeAmenability}. We recommend the reader to look at the proof of the main theorem first (Theorem \ref{Thm=MainImplication})  which contains most conceptual parts of the proof strategy. Then in Sections \ref{Sect=QuasiNormalizer} and \ref{Sect=RelativeAmenability}   we collect the necessary results on the location of normalizers and we prove a result on relative amenability for graph products, which is of independent interest.

\section{Preliminaries}\label{Sect=Preliminaries}

\subsection{Von Neumann algebras, Jones projection, normalizers, strong solidity}
Let $B(H)$ denote the bounded operators on a Hilbert space $H$. The  group von Neumann algebra of a discrete group $G$ is denoted by $\calL(G)$.
  Let $M$ be a von Neumann algebra which in this paper is always {\it assumed to be finite} with faithful normal tracial state $\tau$. Let $M'$ denotes the commutant of $M$.  We call $M$ \textit{diffuse} if $M$ does not contain minimal non-zero projections. For an inclusion $Q \subseteq M$ of finite von Neumann algebras
  we write $\mathbb{E}_Q: M \rightarrow Q$ for the normal  trace preserving conditional expectation. For a unital inclusion $Q \subseteq M$ we let $e_Q: L^2(M) \rightarrow L^2(M)$ be the {\it Jones projection}  which is the orthogonal projection onto $L^2(Q)$.
    Then the von Neumann algebra generated by $M$ and $e_Q$ denoted by $\langle M, e_Q \rangle$ is called the Jones extension, see \cite{jonesIntroductionSubfactors1997}. There exists a normal faithful semi-finite operator valued weight   \cite{HaagerupMathScan1},  \cite{HaagerupMathScan2}
  $\Tr_Q$ from a domain in $\langle M, e_Q \rangle$ to  $M$ for which all operators $x e_Q y, x,y \in M$ are in this domain and such that $\Tr_Q( x e_Q y ) = xy$. We then set the trace $\Tr = \tau \circ \Tr_Q$.

  For a von Neumann subalgebra $A \subseteq M$ we denote
\[
\begin{split}
	\Nor_{M}(A) &:= \{u \in M \text{ unitary} \mid  u A u^\ast = A\}\\
        \qNor_{M}^{(1)}(A) &:= \{x\in M \mid\exists x_1,\ldots, x_n\in  M, {\rm s.t. }
 A x\subseteq \sum_{i=1}^n x_i A \}\\
         \qNor_{M}(A) &:=  \qNor_{M}^{(1)}(A)\cap  \left(\qNor_{M}^{(1)}(A)\right)^*
\end{split}
\]
which are called the normalisers, one-sided quasi-normalisers and  quasi-normalisers respectively. We remark that $\Nor_{M}(A)$ is a group,  $\qNor_{M}^{(1)}(A)$ is an algebra and $\qNor_{M}(A)$ is a $*$-algebra.
Furthermore we see that $\Nor_{M}(A)\subseteq \qNor_{M}(A)$. We have that  $A$ and $A'\cap M$ are contained in the von Neumann algebra $\Nor_{M}(A)''$.

\begin{definition}(Amenability)
	A von Neumann algebra $M \subseteq B(H)$ is called {\it amenable} if there exists a (possibly non-normal) conditional expectation $\mathbb{E}: B(H) \rightarrow M$, i.e. $\mathbb{E}$ is a completely positive map that restricts to the identity on $M$.
\end{definition}

\begin{definition}[Strong solidity]
    A von Neumann algebra $M$ is called {\it strongly solid} if for every diffuse amenable unital von Neumann subalgebra $A \subseteq M$ the normalizer $\Nor_{M}(A)$ generates an amenable von Neumann algebra.
\end{definition}

We remark that amenability passes to von Neumann subalgebras and therefore if $M$ is strongly solid, then so is every von Neumann subalgebra of $M$.

\subsection{Intertwining-by-bimodules} We recall Popa's intertwining-by-bimodules theory \cite{PopaAnnals}, \cite{PopaInventionesI}.

\begin{definition}[Embedding $A \prec_{M}  B$]\label{Dfn=Intertwine}
	For von Neumann subalgebras $A, B \subseteq M$ we will say that \textit{a corner of $A$ embeds in $B$ inside $M$} (denoted $A \prec_{M} B$) if one of the following equivalent statements hold:
 \begin{enumerate}
     \item\label{Item=Intertwine1}  There exists projections $p\in A$, $q\in B$, a  normal $\ast$-homomorphism $\theta: p A p \to q B q$ and a non-zero partial isometry $v\in q M p$ such that
	$\theta(x) v = vx$ for all $x\in p A p$.
 \item\label{Item=Intertwine2} There exists no net of unitaries $(u_i)_i$ in $A$ such that for every $x,y \in M \mathbf{1}_B$ we have that $\Vert \mathbb{E}_B(x u_i y) \Vert_2 \rightarrow 0$.
 \end{enumerate}
\end{definition}

\begin{remark}\label{Rmk=Support}
In  Definition \ref{Dfn=Intertwine} \eqref{Item=Intertwine1}  the range projection $vv^\ast$ is contained in $ \theta(pAp)'$. We may further assume without loss of generality that $q$ equals the support of $\mathbb{E}_B(v v^\ast)$.
\end{remark}

\subsection{Graphs}
Let $\Gamma$ be a simple graph (undirected, no double edges, no edges whose startpoint and endpoint are the same) with vertex set $\Gamma$ and edge set $E\Gamma$. Note the mild abuse of notation as we write $\Gamma$ for both the vertex set as well as the graph, which should not lead to confusion. We will not further use the edge set in our notation.

For a non-empty subset $S\subseteq \Gamma$ we will denote
\begin{align}
	\Link(S) &= \{v\in \Gamma \mid v \textrm{ and } s  \textrm{ share an edge for all } s\in S\}
\end{align}
We note that $\Link(S) = \bigcap_{s\in S}\Link(s)$. Furthermore, for a vertex $v\in \Gamma$ we denote $\Star(v) = \{v\}\cup \Link(v)$. We will write $\Lambda \subseteq \Gamma$ and say that $\Lambda$ is  a subgraph of $\Gamma$ in case the vertex set of $\Lambda$ is a subset of the vertex set of $\Gamma$ and two vertices in $\Lambda$ share an edge if and only if they share an edge in $\Gamma$.  In other words, subgraphs are understood to be complete subgraphs. We remark that our notation $\Link(S)$ always stands for the Link of the set $S$ w.r.t the large graph $\Gamma$ that is fixed, and not w.r.t. subgraphs of $\Gamma$.

\subsection{Coxeter groups}
To every finite simple graph $\Gamma$ we associate a right-angled Coxeter group $W := W_{\Gamma}$ given by the following presentation:
\[
	W_{\Gamma} = \langle \Gamma \mid  v^2 = e \text{ for } v\in \Gamma, vw=wv \text{ for } v,w\in \Gamma \textrm{ sharing an edge} \rangle.
\]
Then $W_\Gamma$ is the graph product over $\Gamma$ with $\mathbb{Z}_2$ as each of the vertex groups \cite{greenGraphProductsGroups1990}. For $\Lambda\subseteq \Gamma$ then $W_\Lambda$ is a subgroup of $W_\Gamma$; we call such a subgroup a {\it special subgroup}.

We will write $M_\Gamma := \cL(W_\Gamma)$. If $\Lambda \subseteq \Gamma$ then $M_\Lambda \subseteq M_\Gamma$.  Throughout the paper $\Gamma$ will always denote a fixed graph of which we shall consider various subgraphs and therefore we sometimes write $W$ for $W_\Gamma$. Note that $W_\Gamma$ consists of words with letters in $\Gamma$ which will typically be denoted by boldface letters. For $\uu \in W$ we denote $\vert \uu \vert$ for the length of $\uu$, i.e. the minimal number of letters needed to represent the word. We will say that a letter $a \in \Gamma$ occurs at the start (resp. end) of  $\uu \in W$ if $\vert a \uu \vert < \vert \uu \vert$ (resp. $\vert \uu a \vert < \vert \uu \vert$).

We state the following proposition from \cite[Theorem 2.15]{CaspersFima} that shows that graph products can be decomposed as amalgamated free products. Note that we only apply this theorem to Coxeter groups and therefore the proposition can also be verified directly by checking that $W_{\Gamma} = W_{\Gamma_1} \ast_{W_{\Gamma_1 \cap \Gamma_2}} W_{\Gamma_2}$.

\begin{proposition}[Decomposition as amalgamated free product]\label{prop:graph-product-decomposition}
	Let $\Gamma$ be a finite simple graph. Fix $v\in V$ and set $\Gamma_{1} = \Star(v)$ and $\Gamma_2 = \Gamma \setminus\{v\}$. Then
	$$M_{\Gamma} = M_{\Gamma_{1}} *_{M_{\Gamma_{1}\cap \Gamma_{2}}} M_{\Gamma_{2}}$$
\end{proposition}

\subsection{Word combinatorics} We collect some lemmas that give control over translations and conjugates of Coxeter subgroups.

\begin{lemma}[Proposition 3.4 of \cite{antolinTitsAlternativesGraph2015}] \label{lemma:intersection-conjugated-group}
    Let   $\Gamma_1\subseteq \Gamma$ be a subgraph and let $\gg\in W_{\Gamma}$.   Then there exists $\Gamma_2 \subseteq \Gamma_1$ and $\hh \in W_{\Gamma_1}$ such that
    \[
       W_{\Gamma_1}\cap \gg W_{\Gamma_1}\gg^{-1} = \hh W_{\Gamma_2}\hh^{-1}
    \]
\end{lemma}

\begin{lemma}\label{Lem=Combinatorics}
Let $\Gamma_1, \Gamma_2$ be subgraphs of $\Gamma$.
Let $\ww \in W_{\Gamma_1}, \uu,\uu' \in W$ be such that (1) $\uu$ and $\uu'$ do not have a letter in $\Gamma_1$ at the start, (2) $\uu$ and $\uu'$ do not have a letter in $\Gamma_2$ at the end, (3)  $\uu^{-1} \ww \uu' \in W_{\Gamma_2}$. Then,  $\ww \in W_{\Gamma_1 \cap \Gamma_2}$,  $\uu = \uu'$ and $\uu$ (and thus $\uu'$)  commutes with $\ww$.
\end{lemma}
\begin{proof}
Suppose that $\ww$ contains a letter $b$ in $\Gamma_1$ which is not contained in $\Gamma_2$, say that we write $\ww = \ww_1 b \ww_2$ as a reduced expression. We may assume that $\ww_1$ does not end on any letters commuting with $b$ by moving those letters into $\ww_2$.  Then as $\uu'$ does not have letters from $\Gamma_1$ at the start we see that $\ww \uu'$ contains the letter $b$; more precisely we may write a reduced expression $\ww \uu' = \ww_1 b \ww_3 \uu''$ where $\ww_3$ is a start of $\ww_2$ and $\uu''$ is a tail of $\uu'$. Since   $\uu^{-1} \ww \uu'$  is contained in $W_{\Gamma_2}$ the letter $b$ cannot occur anymore in its reduced expression. We have $\uu^{-1} \ww \uu' = \uu^{-1} \ww_1 b \ww_3 \uu''$ (possibly non-reduced). Now if a letter at the end of $\uu^{-1}$ deletes the letter $b$ then this would mean that $\uu$ has a letter in $\Gamma_1$ up front (either $b$ itself or a letter from $\ww_1$) which is not possible. We conclude that $\ww \in W_{\Gamma_1 \cap \Gamma_2}$.

Write $\uu = \vv\uu_1$ and $\uu' = \vv\uu_1'$ (both reduced) where $\uu_1,\uu_2\in W$ and where $\vv\in W$ such that $\vv$ commutes with $\ww$. Moreover we can assume that $\uu_1,\uu_2,\vv$ are chosen such that $|\vv|$ is  maximal over all possible choices. Now, suppose that $\uu_1'\not=e$.
Let $d$ be a letter at the end of $\uu_1'$. Then $d\not\in \Gamma_2$ by assumption on $\uu'$ (as $d$ is also at the end of $\uu'$). Now $\uu_1^{-1}\ww\uu_1' = \uu^{-1}\ww\uu'\in W_{\Gamma_2}$, which implies that $d$ is deleted, i.e. $\uu_1^{-1}\ww\uu_1'$ is not reduced. Thus a letter $c$ at the start of $\uu_1$ must delete a letter at the end of $\uu_1^{-1}\ww$. If $c$ deletes a letter from $\ww$ then in particular $c\in \Gamma_1\cap\Gamma_2$ (as $\ww\in W_{\Gamma_1\cap \Gamma_2}$). However, as $\uu'$ does not start with letters from $\Gamma_1$ this implies that $|\vv| \geq 1$. Now, every letter of $\vv$ commutes with the letters from $\ww$ (by assumption on $\vv$). However, not every letter of $\vv$ commutes with $c$, since $c$ is not at the start of $\uu'$. From this we conclude that $c$ is not a letter of $\ww$, a contradiction. We conclude that $c$ is not deleted by a letter from $\ww$, and thus that $c$ must commute with $\ww$, and that $c$ deletes a letter at the end of $\uu_1^{-1}$ i.e. a letter at the start of $\uu_1$. Hence, we can write $\uu_1 = c\uu_2$ and $\uu_1' = c\uu_2'$ (both reduced) for some $\uu_2,\uu_2'\in W$. But then $\uu = \vv c\uu_2$ and $\uu' = \vv c\uu_2'$ and we have that $\vv c$ commutes with $\ww$. This contradicts the maximality of $|\vv|$. We conclude that $\uu_1'=e$. Now as $\uu_1^{-1}\ww = \uu_1^{-1}\ww\uu_1'=\uu^{-1}\ww\uu'$ lies in $W_{\Gamma_2}$ by assumption and as $\ww\in W_{\Gamma_1\cap \Gamma_2}$ we obtain that $\uu_1^{-1}\in W_{\Gamma_2}$. But $\uu_1$ does not end with a letter from $\Gamma_2$ by assumption on $\uu$ (since letters at the end of $\uu_1$ are also at the end of $\uu$). This implies that $\uu_1=e$. This shows $\uu =\vv = \uu'$ and that $\uu$ ($=\vv$) commutes with $\ww$.

\end{proof}

\section{Embeddings of quasi-normalizers in graph products}\label{Sect=QuasiNormalizer}

The main aim of this section is to prove \eqref{Item=Embed2} of Proposition \ref{Prop=LocateNormalizerNew} below. This gives us control over embeddings of normalizers and quasi-normalizers in graph product. The proof is essentially a combination of arguments contained  in \cite{VaesAnnENS} and \cite{IoanaAnnENS} with a number of  modifications that are particular to  graph products.

The following lemma, for $n=1$, is stated and proved in \cite[Lemma 1.4.5]{IoanaPetersonPopaActa} and occurs in many different forms in the literature (see for instance  \cite[Theorem 3.2.2]{VaesAnnENS}).   For completeness we give the proof for arbitrary $n$ here.

\begin{lemma}\label{Lem=TransitiveEmbedding}
Let $A, B_1, \ldots, B_n, Q \subseteq M$ be von Neumann subalgebras with $B_i \subseteq Q$.  Assume that $A \prec_M Q$ but $A \not \prec_M B_i$ for any $i =1, \ldots, n$. Then there exist projections $p \in A, q \in Q$, a partial isometry $v \in q M p$ and a normal $\ast$-homomorphism $\theta: pAp \rightarrow q Q q$ such that $\theta(x) v = vx, x \in pAp$ and such that $\theta(pAp) \not \prec_Q B_i$ for any $i =1, \ldots, n$.
\end{lemma}
\begin{proof}
As $A \prec_M Q$ we may take  $p \in A, q \in Q$, a partial isometry $v \in q M p$ and a normal $\ast$-homomorphism $\theta: pAp \rightarrow q Q q$ such that $\theta(x) v = v x, x \in pAp$. Without loss of generality we may assume that $q$ equals the support projection of $\mathbb{E}_{Q}(v v^\ast)$, see Remark \ref{Rmk=Support}. Now assume that $Q_1 := \theta(pAp) \prec_Q B$ for some $B := B_i$. Then there exist projections $q_1 \in Q_1, r \in B$, a partial isometry $w \in r Q q_1$  and a normal $\ast$-homomorphism $\varphi: q_1 Q q_1 \rightarrow r B r$ such that $\varphi(x) w = wx$ for all $x \in q_1 Q q_1$. Using \cite[I.4.4, Corollary]{DixmierBook}, it follows that we can obtain a projection $p_1 \in pAp$ such that $\theta(p_1) = q_1$. Then the composition $\varphi \circ \theta$ yields a $\ast$-homomorphism $p_1 A p_1 \rightarrow rBr$ and we have $\varphi(\theta(x)) wv = w \theta(x) v = wv x$. Let $wv = u \vert wv \vert$ be the polar decomposition of $wv$ and let $\vert wv \vert^{-1}$ be the operator that acts as $\vert wv \vert^{-1}$ on $\ker(wv)^{\perp}$ and which has $\ker(wv)$ as its kernel.
Then   $\varphi(\theta(x)) u =\varphi(\theta(x)) wv \vert wv \vert^{-1} =  wv x \vert wv \vert^{-1} = wv   \vert wv \vert^{-1} x = u x$ for all $x \in p_1 A p_1$ as $\vert wv \vert \in p_1Ap_1'$.   Finally we claim that $wv$, and therefore $u$, is non-zero. Indeed, if $wv$ would be 0 then   $w \mathbb{E}_Q(vv^\ast) =  \mathbb{E}_Q(wvv^\ast) = 0$ which implies that $w$ is zero as we assumed that the support of $\mathbb{E}_Q(vv^\ast)$ equals $q$, contradiction.  This concludes the proof.
\end{proof}

To proceed we recall the following lemma giving control over quasi-normalizers.

\begin{lemma}[Lemma 2.7 of \cite{chifanAmalgamatedFreeProduct2017} and \cite{PopaInventionesI}]
\label{lemma:embeddability-groups}
Let $G_1\subseteq G$ be countable groups and let $P\subseteq \calL(G_1)$. Assume that
	$P\not\prec_{\calL(G_1)} \calL(G_1\cap g G_1g^{-1})$ for all $g\in G\setminus G_1$. Then if $x\in \calL(G)$ satisfies $xP = \sum_{i=1}^n \calL(G_1)x_i$ for some $x_1,\ldots, x_n\in \calL(G)$ then we obtain $x \mathbf{1}_P\in \calL(G_1)$.
\end{lemma}

The following proposition concerns a result that is well-known for amalgamated free products, see \cite[Theorem 1.2.1]{IoanaPetersonPopaActa}, of which we obtain the analogue for graph products.   Our result also generalizes some recent results in the literature \cite[Lemma 2.9]{CDD1}.

    The second statement of the following proposition should also be compared to \cite[Lemma 9.4]{IoanaAnnENS}.  In the present paper the inclusion $M_\Lambda \subseteq M_\Gamma$ is usually not mixing, but for graph products we still have enough control over the (quasi-)normalizers of subalgebras.

\begin{proposition}\label{Prop=LocateNormalizerNew}
  Let $\Lambda$ be a subgraph of $\Gamma$ and set $M = M_\Gamma$. Let $A \subset M$ be a von Neumann subalgebra and let $P =  \Nor_{M}(A)''$. Let $r \in P \cap P'$ be a projection. The following statements hold true.
  \begin{enumerate}
      \item \label{Item=Embed1} If $A \subseteq M_{\Lambda}$ and $A \not \prec_{M_{\Lambda}}  M_{\widetilde{\Lambda}}$ for all strict subgraphs $\widetilde{\Lambda} \subsetneq \Lambda$ then $\qNor_{M}(A)  \mathbf{1}_A \subset M_{\Lambda\cup \Link(\Lambda)}$.
      \item \label{Item=Embed2} If $r A \prec_{M} M_{\Lambda}$ and  $rA \not\prec_{M}  M_{\widetilde{\Lambda}}$ for all strict subgraphs $\widetilde{\Lambda}\subsetneq \Lambda$ then $r P \prec_{M} M_{\Lambda \cup \Link(\Lambda)}$.
  \end{enumerate}
 \end{proposition}
 \begin{proof}
 \eqref{Item=Embed1} Under the isomorphism
$M_{\Lambda \cup \Link(\Lambda)} \simeq  M_{\Lambda } \otimes   M_{\Link(\Lambda)}$  the inclusion $A \subseteq M_\Lambda \subseteq M_{\Lambda \cup \Link(\Lambda)}$ becomes  $A  \otimes 1 \subset   M_{\Lambda } \otimes   M_{\Link(\Lambda)}$. As $A \not\prec_{M_{\Lambda}} M_{\widetilde{\Lambda}_1}$ for any strict subgraph $\widetilde{\Lambda}_1 \subseteq \Lambda$ it follows from Definition \ref{Dfn=Intertwine} \eqref{Item=Intertwine2} that
\[
A  \otimes 1 \not\prec_{M_{\Lambda} \otimes  M_{\Link(\Lambda)}}   M_{\widetilde{\Lambda}_1 } \otimes   M_{\widetilde{\Lambda}_2 }
\]
for any strict subgraph $\widetilde{\Lambda}_1$ of $\Lambda$ and any (non-strict) subgraph $\widetilde{\Lambda}_2$ of $\Link(\Lambda)$.

    We now aim to apply Lemma \ref{lemma:embeddability-groups} applied to $G = W$, $G_1 = W_{\Lambda \cup \Link(\Lambda)}$   and $P = A$.
    Take $\gg \in W \backslash W_{\Lambda \cup \Link(\Lambda)}$.  By Lemma \ref{lemma:intersection-conjugated-group}  we see that
    \begin{equation}\label{Eqn=Conjugate}
    W_{\Lambda \cup \Link(\Lambda)} \cap \gg W_{\Lambda \cup \Link(\Lambda)} \gg^{-1} = \hh W_{\widetilde{\Lambda}} \hh^{-1},
    \end{equation}
    for some $\hh \in W_{\Lambda \cup \Link(\Lambda)} $ and $\widetilde{\Lambda} \subseteq \Lambda \cup \Link(\Lambda)$.

    We claim that $\widetilde{\Lambda}$ does not contain $\Lambda$. Indeed, if we would have  $\Lambda \subseteq \widetilde{\Lambda}$, then \eqref{Eqn=Conjugate} implies that  $W_{\Lambda \cup \Link(\Lambda)} \cap \gg W_{\Lambda \cup \Link(\Lambda)} \gg^{-1} \supseteq \hh W_{\Lambda} \hh^{-1} = W_\Lambda$. That is $\gg W_{\Lambda \cup \Link(\Lambda)} \gg^{-1}  \supseteq W_\Lambda$. We may write $\gg = \gg_1 \gg_2 \gg_3$ with $\gg_1 \in W_\Lambda, \gg_3 \in W_{\Lambda \cup \Link(\Lambda)}$ and $\gg_2$ having no letters from $\Lambda$ at the start and no letters from $\Lambda \cup \Link(\Lambda)$ at the end. Then $\gg_2 W_{\Lambda \cup \Link(\Lambda)} \gg_2^{-1}  \supseteq W_\Lambda$ and then Lemma \ref{Lem=Combinatorics}   implies that $\gg_2 \in W_{ \Link( \Lambda \cup \Link(\Lambda) ) }\subseteq W_{\Link(\Lambda)}$. It follows that $\gg \in W_{\Lambda \cup \Link(\Lambda)}$,  contradiction. It follows that $\widetilde{\Lambda}$ does not contain  $\Lambda$.

    If $
    A \prec_{M_{\Lambda \cup \Link(\Lambda)}}    M_{\Lambda \cup \Link(\Lambda)} \cap \gg M_{\Lambda \cup \Link(\Lambda)} \gg^{-1}$,
    then it follows that $A \prec_{M_{\Lambda \cup \Link(\Lambda)}}  \hh M_{\widetilde{\Lambda}} \hh^{-1}$, equivalently $A \prec_{M_{\Lambda \cup \Link(\Lambda)}}   M_{\widetilde{ \Lambda }}$. This is excluded by assumption and the previous paragraphs. So  $A \not \prec_{M_{\Lambda \cup \Link(\Lambda)}}    M_{\Lambda \cup \Link(\Lambda)} \cap \gg M_{\Lambda \cup \Link(\Lambda)} \gg^{-1}   $ and the assumptions of Lemma \ref{lemma:embeddability-groups}
    are satisfied. That lemma concludes that   $\qNor_{M}(A) \mathbf{1}_A \subseteq M_{\Lambda\cup \Link(\Lambda)}$.

  \eqref{Item=Embed2} We start by observing that $r$ is in particular central in $A$ which we will use a number of times in the proof. The assumptions imply by Lemma \ref{Lem=TransitiveEmbedding} that there exist projections $p \in r A, q \in M_{\Lambda}$ a non-zero partial isometry $v \in q  M p$ and a normal $\ast$-homomorphism $\theta: p  A p \rightarrow q M_{\Lambda} q$ such that $\theta(x) v = v x$ for all $x \in p  A p$ and such that moreover  $\theta(p  A p) \not \prec_{M_{\Lambda}} M_{\widetilde{\Lambda}}$ for any strict subgraph $\widetilde{\Lambda}$ of $\Lambda$. From  \eqref{Item=Embed1} we see that $\qNor_M(\theta(p A p)) \theta(p) \subseteq M_{\Lambda \cup \Link(\Lambda)}$.

Now take $u \in \Nor_M(A)$. We follow the proof of \cite[Lemma 3.5]{PopaInventionesI} or \cite[Lemma 9.4]{IoanaAnnENS}. Take $z \in A$ a central projection such that $z = \sum_{j=1}^n v_j v_j^\ast$ with $v_j \in A$ partial isometries such that $v_j^\ast v_j \leq p$. Then
\[
pz u pz (p Ap) \subseteq pzuA = pz A u = p A zu \subseteq \sum_{j=1}^n (p A v_j) v_j^\ast u \subseteq \sum_{j=1}^n (p A p) v_j^\ast u,
\]
and similarly $(pAp) pz u pz \subseteq \sum_{j=1}^n u v_j (p A p)$. We conclude that $pz u pz \in \qNor_{pMp}(pAp) $.

Now if $x \in \qNor_{pMp}(p A p)$ then by direct verification we see that  $v x v^\ast \in \qNor_{q M q}(\theta(p A p)) \theta(p) $. It follows that $v p z u pz v^\ast$, with $u \in \Nor_M(A)$ as before, is contained in  $\qNor_{q M q}(\theta(pAp)) \theta(p)$ which was contained in  $M_{\Lambda \cup \Link(\Lambda)}$ by the first paragraph of the proof of   \eqref{Item=Embed2}. We may take the projections $z$ to approximate the central support of $p$ and therefore  $v u v^\ast = v p   u p  v^\ast \in M_{\Lambda \cup \Link(\Lambda)}$.  Hence $v  \Nor_{M}( A)'' v^\ast \subseteq M_{\Lambda \cup \Link(\Lambda)}$.
Set $p_1 = v^\ast v \in pA'p$.  Note that $p_1 \leq p \leq r$. As both $A$ and $A'$ are contained in  $\Nor_{M}( A)''$ we find that $p_1 \in \Nor_{M}( A)''$ (as $p \in A$).
So we have the
 $\ast$-homomorphism $\rho: p_1  \Nor_{M}( A)'' p_1  =  p_1  r \Nor_{M}( A)'' p_1   \rightarrow M_{\Lambda \cup \Link(\Lambda)}: x \mapsto v x v^\ast$ with
 $v \in q M p_1$ and clearly $\rho(x) v = v x$. We conclude that    $r \Nor_{M}( A)'' \prec_M  M_{\Lambda \cup \Link(\Lambda)}$.
 \end{proof}

\section{Relative amenability}\label{Sect=RelativeAmenability}

 We start with introducing the notion of relative amenability.
Let $M$ be a finite von Neumann algebra with faithful tracial state $\tau$ and let $Q, P  \subset M$ be   von Neumann subalgebras and assume the inclusion $Q \subseteq M$ is unital.     Let
\[
T_{Q}: L^1(\langle M, e_Q \rangle, \Tr) \rightarrow L^1(M, \tau),
\]
be the unique map defined through $\tau(T_{Q}(y) x) = \Tr(y x)$ for all $y \in L^1(\langle M, e_Q \rangle , \Tr), x \in M$. Then $T_Q$ is the predual of the inclusion map $M \subset \langle M, e_Q \rangle$ and thus is contractive and preserves positivity.
 For the following definition of relative amenability we refer to \cite[Proposition 2.4]{PopaVaesActa}.

 \begin{definition}\label{Dfn=RelativeAmenableFirst}
We say that $P$ is amenable relative to $Q$ inside $M$ if  there exists a $P$-central positive functional on
$\mathbf{1}_P \langle M, e_Q \rangle \mathbf{1}_P$ that restricts to the trace $\tau$ on $\mathbf{1}_P M \mathbf{1}_P$.
\end{definition}

\begin{remark}\label{Rmk=NonUnitalAmenable}
Assume the inclusion $P \subseteq M$ is not unital. Let $p = \mathbf{1}_M - \mathbf{1}_P$. Set $\widetilde{P} = P \oplus \mathbb{C} p$ which is a unital subalgebra of $M$. We claim: $P$ is amenable relative to $Q$ inside $M$ if and only if $\widetilde{P}$ is amenable relative to $Q$ inside $M$.  Indeed, for the if part,   choose a $\widetilde{P}$-central positive functional $\widetilde{\Omega}$  on $\langle M, e_Q \rangle$ that restricts to $\tau$ on $M$. Set $\Omega$ to be the restriction of $\widetilde{\Omega}$ to $\mathbf{1}_P \langle M , e_Q \rangle \mathbf{1}_P$ which then clearly witnesses relative amenability of $P$. For the only if part, let $\Omega$ be  a
$P$-central positive functional  on $\mathbf{1}_P \langle M, e_Q \rangle \mathbf{1}_P$ that restricts to $\tau$ on  $\mathbf{1}_P M \mathbf{1}_P$ then we set $\widetilde{\Omega}(x) = \Omega(p x p ) + \widetilde{\tau}( (\mathbf{1}-p) x (\mathbf{1}-p) )$ for any positive functional $\widetilde{\tau}$ extending $\tau$ from $ (\mathbf{1}-p) M (\mathbf{1}-p)$ to $(\mathbf{1}-p) \langle M , e_Q \rangle (\mathbf{1}-p)$.  Clearly $\widetilde{\Omega}$ witnesses the relative amenability of $\widetilde{P}$.
\end{remark}

\begin{proposition}[Proposition 2.4 of \cite{PopaVaesActa}]\label{Prop=RelativeAmenable}
Assume $P, Q \subseteq M$ are unital von Neumann subalgebras. Then  $P$ is amenable relative to $Q$ inside $M$ if and only if
there exists a net $(\xi_j)_j \in L^2( \langle M, e_Q \rangle, \Tr )^+$ such that:
\begin{enumerate}
    \item   $0\leq T_Q(\xi_j^2) \leq \mathbf{1}$ for all   $j$ and $\lim_j \Vert T_Q(\xi_j^2) - \mathbf{1} \Vert_1 = 0$;
    \item For all $y \in P$ we have $\lim_j \Vert y \xi_j -\xi_j y \Vert_2 = 0$.
\end{enumerate}
\end{proposition}

The aim of this section is to prove \cref{Thm=Square} for which we need a number of auxiliary lemmas. Before we state the following lemma, we recall that we identify the vertices of the graph $\Gamma$ with the generators of the vertex groups. Recall further that for $\Lambda \subseteq \Gamma$ we set $\Link(\Lambda) = \bigcap_{v \in \Lambda} \Link(v)$. Then for $\vv \in W$ we let $\Link(\vv) = \Link(\Lambda_\vv)$ where $\Lambda_\vv$ is the set of all letters that occur in $\vv$. Alternatively, $\Link(\vv)$ can be described as the set of all $w \in \Gamma \backslash \Lambda_\vv$ such that $w \vv = \vv w$.

\begin{lemma}\label{Lem=ConditionNest}
Let $\Gamma_1, \Gamma_2$ be subgraphs of $\Gamma$. Let $\vv, \uu \in W$ and write $\vv = \vv_l \vv_c \vv_r, \uu = \uu_l \uu_c \uu_r$ where $\vv_l, \uu_l \in W_{\Gamma_1}, \vv_r, \uu_r \in W_{\Gamma_2}$
and $\vv_c, \uu_c$ has no letters from $\Gamma_1$ at the start and no letters from $\Gamma_2$ at the end. Then, for $x \in M$,
\[
  \mathbb{E}_{M_{\Gamma_2}}(  \lambda_{\vv}^\ast \mathbb{E}_{M_{\Gamma_1} }(x) \lambda_{\uu} )
=
\left\{
\begin{array}{ll}
\lambda_{\vv_r}^\ast
\mathbb{E}_{ M_{\Gamma_1} \cap M_{\Gamma_2} \cap M_{ \Link(\uu_c) } }(  \lambda_{\vv_l}^\ast x \lambda_{\uu_l} ) \lambda_{\uu_r}  & \textrm{ if } \uu_{c} = \vv_c, \\
0   &   \textrm{otherwise}.
\end{array}
\right.
\]
\begin{proof}
We have
\begin{equation}\label{Eqn=ExpectationNest}
 \mathbb{E}_{  M_{\Gamma_2}   }(  \lambda_{\vv}^\ast \mathbb{E}_{  M_{\Gamma_1}  }( x) \lambda_{\uu} ) =
\lambda_{\vv_r}^\ast  \mathbb{E}_{ M_{\Gamma_2} }(  \lambda_{\vv_c}^\ast \mathbb{E}_{ M_{\Gamma_1}  }(   \lambda_{\vv_l}^\ast x \lambda_{\uu_l}   ) \lambda_{\uu_c} ) \lambda_{\uu_r}.
\end{equation}
In case $\uu_c \not = \vv_c$
\cref{Lem=Combinatorics} shows that $\mathbb{E}_{M_{\Gamma_2}}(  \lambda_{\vv_c}^\ast \mathbb{E}_{M_{\Gamma_1}}(  \: \cdot \: ) \lambda_{\uu_c} )$ is the zero map. In case $\uu_c = \vv_c$ \cref{Lem=Combinatorics} shows that
 \[
 \begin{split}
 \mathbb{E}_{  M_{\Gamma_2}  }(  \lambda_{\vv_c}^\ast \mathbb{E}_{   M_{\Gamma_1}  }(  \: \cdot \: ) \lambda_{\uu_c} ) &= \mathbb{E}_{  M_{\Gamma_2}  }(  \lambda_{\vv_c}^\ast \mathbb{E}_{M_{\Gamma_1} \cap M_{\Gamma_2} \cap  M_{
 \Link(\uu_c) }}(  \: \cdot \: ) \lambda_{\uu_c} ).\\
 &= \mathbb{E}_{ M_{\Gamma_1} \cap M_{\Gamma_2} \cap  M_{
 \Link(\uu_c) }}(  \: \cdot \: )
 \end{split}
 \]
 So continuing \eqref{Eqn=ExpectationNest} we get
\[
 \mathbb{E}_{ M_{\Gamma_2}  }(  \lambda_{\vv}^\ast \mathbb{E}_{ M_{\Gamma_1} }( x ) \lambda_\uu ) =
 \delta_{\uu_{c},\vv_c}\lambda_{\vv_r}^\ast
\mathbb{E}_{ M_{\Gamma_1} \cap M_{\Gamma_2} \cap  M_{
 \Link(\uu_c) } }(  \lambda_{\vv_l}^\ast x \lambda_{\uu_l} ) \lambda_{\uu_r},
\]
which is the desired equality.
\end{proof}
\end{lemma}

\begin{remark}
In Lemma \ref{Lem=ConditionNest} in   the decomposition  $\vv = \vv_l \vv_c \vv_r$  the word  $\vv_c$ is  unique and therefore the statement of the lemma is well-defined. Note that $\vv_l$ and  $\vv_r$ may not be unique.   The same holds for $\uu$.
\end{remark}

Let $M,Q,N$ be tracial von Neumann algebras and let $_MH_Q$ and $_QK_N$ be bimodules. Recall that a vector $\xi\in H$ is called \textit{right $Q$-bounded} if there exists $C>0$ s.t. $\|\xi y\|\leq C\|y\|$ for all $y\in Q$.
For a right $Q$-bounded vector $\xi\in H$ we define $L(\xi) \in B(L^2(Q,\tau), H)$ as $L(\xi)x = \xi x$ where $x \in Q$.
Then, for right $Q$-bounded vectors $\xi,\eta\in H$ we have that $L(\eta)^*L(\xi) \in Q$.
 \noindent We denote by $H_0\subseteq H$ the subspace of all right $Q$-bounded vectors.
We equip the algebraic tensor product $H_0\otimes K$ with the (possibly degenerate) inner product
\begin{align}
	\langle \xi_1 \otimes \eta_1,\xi_2\otimes \eta_2\rangle_{H_0\otimes_{Q} K} := \langle L(\xi_2)^*L(\xi_1)\eta_1,\eta_2\rangle_K.
\end{align}
The Connes tensor product $H\otimes_{Q} K$ is the Hilbert space obtained from $H_0\otimes_{\alg} K$ by quotienting out the degenerate part and taking a completion. The Hilbert space $H\otimes_{Q} K$ is a $M$-$N$ bimodule with the action
$$x\cdot(\xi \otimes_{Q} \eta)\cdot y = (x\xi) \otimes_{Q} (\eta y).$$

\begin{remark}\label{remark:calculation-connes-fusion}
    We calculate the operator $L(\xi_2)^*L(\xi_1)$ for certain bimodules and vectors $\xi_1,\xi_2 \in H_0$. Let $(M,\tau)$ be a tracial von Neumann algebra and let $P,Q\subseteq M$ be  von Neumann subalgebras   with $Q$ unital. Consider the bimodule $_PL^2(M,\tau)_{Q}$. Let $x,y\in M$. Then $x,y$ are right $Q$-bounded and thus $L(x),L(y):L^2(Q,\tau)\to L^2(M,\tau)$ are well-defined. We calculate $L(x)^*L(y)$. For $q_1,q_2\in L^2(Q,\tau)$ we have
    \begin{align}
        \langle L(x)^*L(y)q_1, q_2\rangle =
        \langle yq_1, xq_2\rangle
        = \tau(q_2^*x^*yq_1)
        = \tau(q_2^*\EE_{Q}(x^*y)q_1)
        =\langle \EE_{Q}(x^*y)q_1, q_2\rangle.
    \end{align}
    Thus $L(x)^*L(y) = \EE_{Q}(x^*y)$.

    Let $R\subseteq M$ be a  unital von Neumann subalgebra and let $N = \langle M,e_{R}\rangle$, where $e_R$ denotes the Jones projection of the inclusion $R\subseteq M$. We consider the bimodule $ _PL^2(N,\Tr)_{Q}$.
    For $x,x',y,y'\in M$ we have that $xe_{R}y$ and $x'e_{R}y'$ are right $Q$-bounded vectors as they are elements in $N$. We calculate $L(xe_{R}y)^*L(x'e_{R}y')$. For $q_1, q_2 \in Q$ we have,
    \[
    \begin{split}
        \langle L(xe_{R}y)^*L(x'e_{R}y')q_1,q_2\rangle&=
        \langle x'e_{R}y'q_1,xe_{R}yq_2\rangle \\
        &= \Tr(q_2^*y^*e_{R}x^*x'e_{R}y'q_1)\\
        &= \Tr(q_2^*y^*\EE_{R}(x^*x')e_{R}y'q_1)\\
        &= \tau(q_2^*y^*\EE_{R}(x^*x')y'q_1)\\
        &= \tau(\EE_{Q}(q_2^*y^*\EE_{R}(x^*x')y'q_1))\\
        &= \tau(q_2^*\EE_{Q}(y^*\EE_{R}(x^*x')y')q_1)\\
        &=\langle \EE_{Q}(y^*\EE_{R}(x^*x')y')q_1,q_2\rangle.
    \end{split}
    \]
    Thus we obtain $L(xe_{R}y)^*L(x'e_{R}y') =\EE_{Q}(y^*\EE_{R}(x^*x')y').$
\end{remark}

The proof of the following theorem follows \cite[Proposition 2.7]{PopaVaesActa} but in our case the subalgebras are not regular.

\begin{theorem}\label{Thm=Square}
Let $\Gamma_1, \Gamma_2$ be subgraphs of $\Gamma$.
 Suppose that $P \subseteq M := M_\Gamma = \cL(W)$ is a von Neumann algebra that is amenable relative to $Q_i := M_{\Gamma_i} = \cL(W_{\Gamma_i})$ inside $M$ for  $i = 1,2$.   Then $P$ is amenable relative to $Q_1\cap Q_2 = M_{\Gamma_1 \cap \Gamma_2}$ inside $M$.
\end{theorem}
\begin{proof}
By \cref{Rmk=NonUnitalAmenable} we may assume without loss of generality that the inclusion $P \subseteq M$ is unital and use the characterisation of relative amenability given by \cref{Prop=RelativeAmenable}.

  As before, let  $T_i = T_{Q_i}: L^1( \langle M, e_{Q_i}\rangle) \rightarrow L^1(M)$  be the contraction determined by $\tau(T_i(S) x) = \Tr_i(Sx)$ for $S \in L^1( \langle M, e_{Q_i} \rangle ) $ and $x \in M$. Since $P$ is amenable relative to $Q_i$, \cref{Prop=RelativeAmenable}  implies the existence of nets $(\mu_{j}^i)_j$ in $L^2(\langle M, e_{Q_i} \rangle )^+$ satisfying
\begin{equation}\label{Eqn=AmenableTensor}
0 \leq    T_i(    (\mu_{j}^i)^2 ) \leq \mathbf{1}, \qquad \Vert T_i(  (\mu_j^i)^2) - \mathbf{1}\Vert_1 \rightarrow 0, \qquad \Vert y \mu_{j}^i - \mu_{j}^i y \Vert_2 \rightarrow 0, \textrm{ for all } y \in P,
\end{equation}
where the limits are taken over $j$. Consider the $M$-$M$ bimodule
\[
H = L^2( \langle M, e_{Q_1} \rangle ) \otimes_M L^2( \langle M, e_{Q_2} \rangle ).
\]
\noindent {\it Claim:} As in \cite{PopaVaesActa} we claim that tensor products $\mu_{j} := \mu^1_{j_1} \otimes\mu^2_{j_2}\in H$ for certain $j = (j_1, j_2)$ can be combined into a net such that
\[
\Vert y \mu_j - \mu_j y \Vert \rightarrow 0, \qquad \vert \langle x \mu_j, \mu_j \rangle - \tau(x) \vert \rightarrow 0,
\]
 for all $y\in P$, $x\in M$, where the limit is taken over $j$.
 Let us now prove this claim in the next paragraphs which repeats the argument used in \cite[Proposition 2.4]{PopaVaesActa}.

\vspace{0.3cm}

\noindent {\it Proof of the claim.} Take $\mathcal{F} \subseteq P, \mathcal{G} \subseteq M$ finite and let $\varepsilon > 0$.
Set $\mathcal{G}^1 := \mathcal{G}$ and fix $j_1$ such that $\Vert y \mu^1_{j_1} - \mu^1_{j_1} y \Vert_2 \leq \varepsilon$  for all $y \in \mathcal{F}$
and $\vert \langle x \mu_{j_1}^1, \mu_{j_1}^1 \rangle - \tau(x) \vert \leq \varepsilon$ for all $x \in \mathcal{G}^1$.
As $0 \leq   T_i(    (\mu_{j_1}^i)^2 ) \leq \mathbf{1}$ and as $T_i$ preserves positivity, it follows that for $x \in M$ the element $T_1(\mu_{j_1}^1 x \mu_{j_1}^1) \in L^1(M)$ is bounded in the uniform norm and thus belongs to $M$.
Set $\mathcal{G}^2 := T_1( \mu_{j_1}^1 \mathcal{G}^1 \mu_{j_1}^1) \subseteq M$, which is finite. We may proceed from $\mathcal{F}$ and $\mathcal{G}^2$ to find $j_2$ such that  $\Vert y \mu^2_{j_2} - \mu^2_{j_2} y \Vert_2 \leq \varepsilon$  for all $y \in \mathcal{F}$ and $\vert \langle x \mu_{j_2}^2, \mu_{j_2}^2 \rangle - \tau(x) \vert \leq \varepsilon$ for all $x \in \mathcal{G}^2$. Put $j= (j_1,j_2)$ and set $\mu_j =  \mu^1_{j_1} \otimes \mu^2_{j_2}$.

It follows by the triangle inequality that for $y \in \mathcal{F}$ we have
\[
\Vert y \mu - \mu y \Vert \leq \Vert (y\mu^{1}_{j_1}-\mu^1_{j_1}y) \otimes_M    \mu^{2}_{j_{2}}  \Vert + \Vert \mu^{1}_{j_{1}} \otimes_M (y\mu^{2}_{j_2}-\mu^2_{j_2}y) \Vert \\
\leq  2 \varepsilon.
\]
Now, by construction of the sets $\mathcal{G}^i$ and the vectors $\mu^i_{j_i}$ we see that for  $x \in \mathcal{G}$ that
\begin{align*}
		\vert \langle x \mu, \mu \rangle - \tau( x)  \vert &\leq |\langle x\mu,\mu\rangle - \langle x\mu_{j_2}^2,\mu_{j_2}^2\rangle| +
		|\langle x\mu_{j_2}^2,\mu_{j_2}^2\rangle - \tau(x)|\\
		&\leq |\langle T_1(\mu_{j_1}^1 x \mu_{j_1}^1)\mu_{j_2}^2,\mu_{j_2}^2\rangle - \langle x\mu_{j_2}^2,\mu_{j_2}^2\rangle| +
		|\langle x\mu_{j_2}^2,\mu_{j_2}^2\rangle - \tau(x)|\\
		&\leq 2\epsilon
\end{align*}

Taking $j = j(\mathcal{F}, \mathcal{G})$ with increasing sets $\mathcal{F}$ and $\mathcal{G}$ as before gives a net of vectors $\mu_j \in H$ with the property that
\[
\Vert y \mu_{j} - \mu_j y \Vert \rightarrow 0, \qquad \vert \langle x \mu_j, \mu_j \rangle - \tau(x) \vert \rightarrow 0
\]
for all $y\in P$ and $x\in M$. This proves the claim.

\vspace{0.3cm}

\noindent {\it Remainder of the proof.} The net $(\mu_j)_j$ in particular shows that the bimodule $_M L^2(M)_P$ is weakly contained in $_M H_P$.
Denote $Q_0 := Q_1\cap Q_2$ and put
	$K :=L^2(M,\tau)\otimes_{Q_0} L^2(M,\tau)$.  We show that  $_{M} H_{M}$ is contained in the bimodule $_{M} L^2(M)\otimes_{Q_0} K_{M}$. Let
 $$V = \{\vv\in W \mid \vv \text{ does not start with letters from } \Gamma_1 \text{ and does not end with letters from } \Gamma_2\}.$$
Observe that the subspace
	$$H_0:=\Span\{x e_{Q_1} \lambda_{\vv} \otimes_{M}   e_{Q_2} z \mid  x, z \in M, \vv \in V\}$$ is dense in $H$. Indeed, it is clear that the span of vectors $x e_{Q_1} \lambda_{\vv} \otimes_{M}   e_{Q_2} z$ with $x,z\in M$ and $\vv\in W$  is dense in $H$. For $\vv\in W$ we can write $\vv=\vv_l\vv_c\vv_r$ with $\vv_l\in W_{\Gamma_1}$, $\vv_r\in W_{\Gamma_2}$ and $\vv_c\in V$. Therefore we obtain  $$x e_{Q_1} \lambda_{\vv} \otimes_{M}   e_{Q_2} z = (x\lambda_{\vv_l}) e_{Q_1} \lambda_{\vv_c}\otimes_{M}   e_{Q_2} (\lambda_{\vv_r}z)\in H_0$$
	which shows density of $H_0\subseteq H$. Define  $U:H_0\to L^2(M)\otimes_{Q_0} K$ as
	\begin{align*}
		x e_{Q_1} \lambda_{\vv} \otimes_{M}   e_{Q_2}z  \mapsto x  \otimes_{Q_0} \lambda_{\vv} \otimes_{Q_0} z &\quad \text{ for } x,z\in M, \vv\in V
	\end{align*}
    We now use \cref{Lem=ConditionNest} and the calculations from \cref{remark:calculation-connes-fusion} to show that $U$ is isometric. Indeed, for $x,x',z,z'\in M$ and $\vv,\uu\in V$ we find,
	\[
 \begin{split}
		\langle x' \otimes_{Q_0} \lambda_{\uu} \otimes_{Q_0} z', x\otimes_{Q_0} \lambda_{\vv}\otimes_{Q_0} z\rangle_{L^2(M)\otimes_{Q_0} K}
		&= \langle \EE_{Q_0}(x^*x')\lambda_{\uu}\otimes_{Q_0} z',\lambda_{\vv}\otimes_{Q_0} z\rangle_K\\
		&= \langle\EE_{Q_0}(\lambda_{\vv}^*\EE_{Q_0}(x^*x')\lambda_{\uu})z',z\rangle\\
		&= \delta_{\vv,\uu}\langle\EE_{Q_0\cap \Link(\vv)}(x^*x')z',z\rangle\\
		&=	\langle \EE_{Q_2}(\lambda_{\vv}^*\EE_{Q_1}(x^*x')\lambda_{\uu})z', z\rangle \\
		&=	\langle T_{Q_2}(\EE_{Q_2}(\lambda_{\vv}^*\EE_{Q_1}(x^*x')\lambda_{\uu}) e_{Q_2}z'), z\rangle \\
		&=	\langle T_{Q_2}(e_{Q_2}\lambda_{\vv}^*\EE_{Q_1}(x^*x')\lambda_{\uu}e_{Q_2}z'), z\rangle \\
		&=	\langle \lambda_{\vv}^*\EE_{Q_1}(x^*x')\lambda_{\uu}e_{Q_2}z', e_{Q_2} T_{Q_2}^*(z)\rangle \\
		&=	\langle \EE_{M_{\Gamma}}(\lambda_{\vv}^*\EE_{Q_1}(x^*x')\lambda_{\uu})e_{Q_2}z', e_{Q_2}z\rangle\\
       &=	\langle x'e_{Q_1}\lambda_{\uu} \otimes_{M_{\Gamma}} e_{Q_2}z', xe_{Q_1}\lambda_{\vv} \otimes_{M_{\Gamma}} e_{Q_2}z\rangle_H
	\end{split}
 \]
	Thus $U$ extends to an isometry $H\to L^2(M)\otimes_{Q_0} K$, which clearly is $M$-$M$-bimodular.
	We have shown that $_{M}L^2(M)_P$    is weakly contained in $_{M}L^2(M) \otimes_{Q_0} K_P$, which by \cite[Proposition 2.4 (3)]{PopaVaesActa} means that $P$ is amenable relative to $Q_0 = Q_1\cap Q_2$.
\end{proof}

\section{Main theorem: classifying strong solidity for right-angled Coxeter groups}

In this section we collect our main results. The proof is strongly based on the following alternatives for amalgamated free product decompositions.

\begin{theorem}[Theorem A of \cite{VaesPrims}]
\label{thm:prelim:alternatives}
    Let $(N_1,\tau_1)$,$(N_2,\tau_2)$ be tracial von Neumann algebras with a common von Neumann subalgebra $B\subseteq N_i$ satisfying $\tau_1|_{B} = \tau_2|_{B}$ and denote $N:= N_1 \ast_{B} N_2$ for the amalgamated free product. Let $A\subseteq N$ be a von Neumann subalgebra  that is amenable relative  to  $N_1$ or $N_2$ inside $N$. Put $P = \Nor_{\mathbf{1}_{A}N\mathbf{1}_{A}}(A)''$. Then  at least one of the following is true:
    \begin{enumerate}[(i)]
        \item $A\prec_{N} B$,
        \item $P \prec_{N} N_i$ for some $i=1,2$,
        \item $P$ is amenable relative to $B$ inside $N$.
    \end{enumerate}
\end{theorem}

Recall also the following observation.

\begin{proposition}\label{Prop=EmbedSolid}
Let $N \subseteq M$ be a von Neumann subalgebra and assume $N$ is strongly solid.
Let $A \subseteq M$ be diffuse amenable and let $P = \Nor_M(A)''$ and let $z \in P \cap P'$ be a non-zero projection. Assume that $zP \prec_M N$. Then $zP$ has an amenable direct summand.
\end{proposition}
\begin{proof}   We follow \cite[Proof of Corollary C]{VaesPrims}. As $zP \prec_M N$, using the characterization \cite[Theorem 3.2.2]{VaesAnnENS}, (following \cite{PopaInventionesI}), there exists a non-zero projection $p \in M_n(\mathbb{C}) \otimes N$ and a normal unital $\ast$-homomorphism $\varphi: zP \rightarrow p (M_n(\mathbb{C} ) \otimes N) p$. So $\varphi(Az)$ is a diffuse amenable von Neumann subalgebra of $M_n(\mathbb{C}) \otimes N$  and  $\widetilde{P} = \Nor_{ p (M_n(\mathbb{C} ) \otimes N) p}(\varphi(Az))''$  contains $\varphi(Pz)$. As $N$ is strongly solid, so is its amplification $p (M_n(\mathbb{C}) \otimes N)  p$ \cite[Proposition 5.2]{HoudayerMathAnn} and hence $\widetilde{P}$ is amenable. So  $\varphi(Pz)$ is amenable and therefore  $Pz$ contains an amenable direct summand.
\end{proof}

Now  let $K_{2,3}$ be the complete bipartite graph of $2+3$ vertices. More precisely,  the graph with vertices $a_1, a_2, b_1, b_2, b_3$  and with edges between each $a_i$ and $b_j$ for all $i,j$. We let $K_{2,3}^+$ be $K_{2,3}$ but with one extra edge connecting $b_1$ and $b_2$. Let $L$ be the graph with 3 vertices and no edges and let $L^+$ be the graph with 3 vertices and 1 edge. So $L$ is a subgraph of $K_{2,3}$ and $L^+$ is a subgraph of $K_{2,3}^+$.

We first characterize amenability.

 \begin{theorem}\label{Thm=AmenableCoxeter}
 Let $W$ be a right-angled Coxeter group. Then the following are equivalent
 \begin{enumerate}
     \item\label{Item=Amen1} $W$ is non-amenable
    \item\label{Item=Amen2} $W$ contains $\FF_2$ as a subgroup
    \item\label{Item=Amen3} $\Gamma$ contains $L$ or $L^+$ as a subgraph.
 \end{enumerate}
 \end{theorem}
\begin{proof}
\eqref{Item=Amen3} implies \eqref{Item=Amen2}. Suppose $\Gamma$ contains $L$ or $L^+$ as a subgraph. Then there are $u,v,w\in \Gamma$ such that $u,w\not\in \Link(v)$. But then $\{uv,wv\}$ generate $\FF_{2}$, the free group of two generators.

\eqref{Item=Amen2} implies \eqref{Item=Amen1}. Suppose $\FF_2\subseteq W$. Then since $\FF_2$ is non-amenable also $W$ is not amenable.

\eqref{Item=Amen1} implies \eqref{Item=Amen3}.
We prove the implication by induction to the size of the graph. Let $\Gamma$ be a non-empty graph, and assume the claim has been proven for strict subgraphs of $\Gamma$. Suppose $L$ and $L^+$ are not subgraphs of $\Gamma$. For every $v\in \Gamma$ we have that $\Gamma_{v}:=\Gamma\setminus \Star(v)$ contains at most $1$ element. Indeed, if $u,w\in \Gamma_{v}$ with $u\not=w$ then $\{u,v,w\}$ is a subgraph of $\Gamma$ that is either isomorphic to $L$ (when $u\not\in \Link(w)$) or to $L^+$ (when $u\in \Link(w)$).

Fix $v\in \Gamma$ and set $\Lambda_1 := \{v\}\cup \Gamma_{v}$ and $\Lambda_2 := \Gamma\setminus \Lambda_1$. We show $W_{\Gamma} = W_{\Lambda_1}\times W_{\Lambda_2}$. Indeed, if $\Gamma_{v}$ is empty, then $\Star(v) = \Gamma$ so that $\Lambda_1 = \{v\} \subseteq  \Link(\Gamma\setminus \{v\}) = \Link(\Lambda_2)$ from which the decomposition follows. Now suppose $\Gamma_v$ is non-empty, then $\Gamma_{v} = \{w\}$ for some $w\in \Gamma\setminus \Star(v)$ so that $v\not\in \Star(w)$ and thus $v\in \Gamma_{w}$. Hence $\Gamma_{w} = \{v\}$ since it contains at most $1$ element. Then as $\Gamma = \Gamma_{v}\cup \Star(v) = \{w\}\cup \Star(v)$ and $\Gamma = \Gamma_{w}\cup \Star(w) = \{v\} \cup \Star(w)$ we obtain $\Lambda_2=\Gamma\setminus\{v,w\}\subseteq \Star(v)\cap \Star(w) \subseteq \Link(\{v,w\}) = \Link(\Lambda_1)$ and the decomposition follows.

Observe that either $W_{\Lambda_1}=\ZZ_2$ (when $\Gamma_v$ is empty) or $W_{\Lambda_1} =  \ZZ_2*\ZZ_2$ (when $\Gamma_{v}$ is non-empty). In either case we obtain that $W_{\Lambda_1}$ is amenable (as it has polynomial growth). Furthermore, as $\Lambda_2\subseteq \Gamma\setminus\{v\}$ is a strict subgraph of $\Gamma$, and as $\Lambda_2$ does not contain $L$ or $L^+$ as a subgraph (as this holds for $\Gamma$) we obtain by the induction hypothesis that $W_{\Lambda_2}$ is amenable as well. From the decomposition for $W_{\Gamma}$ it now follows that $W_{\Gamma}$ is amenable.
\end{proof}

 The following is the main theorem of this paper.

\begin{theorem}\label{Thm=MainImplication}
Let $W=W_{\Gamma}$ be a right-angled Coxeter group with graph $\Gamma$. Suppose $\Gamma$ does not contain $K_{2,3}$ or $K_{2,3}^+$ as a subgraph. Then $\cL(W)$ is strongly solid.
\end{theorem}
\begin{proof}

The proof is based on induction to the number of vertices of the graph. The statement clearly holds when $\Gamma=\emptyset$ since in that case $\cL(W_{\Gamma}) =\CC$ is strongly solid.

\vspace{0.3cm}

\noindent {\it Induction.}
Let $\Gamma$ be a non-empty graph, and assume by induction that \cref{Thm=MainImplication} is proved for any strictly smaller subgraph of $\Gamma$, i.e. with fewer vertices.
Assume  $K_{2,3}$ and $K_{2,3}^+$ are not subgraphs of $\Gamma$. We shall show that $M_\Gamma := \cL(W_\Gamma)$ is strongly solid. Let $A \subseteq M$ be diffuse and amenable and denote $P = \Nor_M(A)''$. We will show that $P$ is amenable.
We put $\Gamma' := \bigcap_{v\in \Gamma} \Star(v)$ which is a complete, possibly empty graph. Since $\Link(\Gamma') = \Gamma\setminus \Gamma'$ we have $W = W_{\Gamma'}\times W_{\Gamma\setminus \Gamma'}$ and so $M_\Gamma = M_{\Gamma'} \otimes M_{\Gamma\setminus \Gamma'}$. As $W_{\Gamma'}$ is finite (by completeness of $\Gamma'$) we have that $M_{\Gamma'}$ is isomorphic to a subalgebra of $\Mat_{N}(\CC)$, the space of $N\times N$ matrices with $N := |W_{\Gamma'}|$ being the size of $W_{\Gamma'}$. Moreover, if $\Gamma'$ is non-empty, then $M_{\Gamma\setminus \Gamma'}$ is strongly solid by the induction hypothesis, so that it follows from \cite[Proposition 5.2]{HoudayerMathAnn} and the fact that strong solidity passes to subalgebras that $M_\Gamma  \subseteq \Mat_{N}(\CC) \otimes M_{\Gamma\setminus \Gamma'}$ is strongly solid as well.

Hence we may assume that $\Gamma'=\emptyset$. Thus for all $v\in \Gamma$ we obtain $\Star(v)\not=\Gamma$ (since otherwise $v\in \Gamma'$).
Pick $v \in \Gamma$ and set $\Gamma_1:=\Star(v)$ and  $\Gamma_2 := \Gamma \backslash v$. By \cref{prop:graph-product-decomposition} we can decompose  $M_\Gamma = M_{\Gamma_1} \ast_{ M_{\Gamma_1 \cap \Gamma_2}} M_{\Gamma_2}$. Moreover, as $\Gamma_1$, $\Gamma_2$ and $\Gamma_{1}\cap \Gamma_2$ are strict subgraphs of $\Gamma$ we obtain by our induction hypothesis that $M_{\Gamma_1}$, $M_{\Gamma_2}$ and $M_{\Gamma_1\cap \Gamma_2}$ are strongly solid.

Let $z\in P\cap P'$ be a central  projection such that $z P$ has no amenable direct summand. Note that $zP \subseteq \Nor_{z M_\Gamma z}( zA)''$. As $zA$ is amenable, it is amenable relative to $M_{\Gamma_1}$ in $M_{\Gamma}$. Therefore by \cref{thm:prelim:alternatives} at least one of the following three holds.
\begin{enumerate}
 \item\label{Case1} $z A \prec_{M_\Gamma} M_{\Gamma_1 \cap \Gamma_2}$,
 \item\label{Case2} There is $i \in \{ 1, 2\}$ such that $z P \prec_{M_\Gamma} M_{\Gamma_i}$,
 \item\label{Case3} $z P$ is amenable relative to $M_{\Gamma_1 \cap \Gamma_2}$ inside $M_\Gamma$.
\end{enumerate}
We now analyse each of the cases.

\vspace{0.3cm}

\noindent {\it Case \eqref{Case2}.} In Case \eqref{Case2} we have that Proposition \ref{Prop=EmbedSolid} together with the induction hypothesis shows that $zP$ has an amenable direct summand in case $z \not = 0$. This is a contradiction so we conclude $z=0$ and hence $P$ is amenable.

\vspace{0.3cm}

\noindent {\it Case \eqref{Case1}.}  In Case \eqref{Case1} we  first prove the following claim.

\vspace{0.3cm}

\noindent {\it Claim:} At least one of the following holds:
\begin{enumerate}
    \item[(a)] There is a vertex $w \in \Gamma \backslash \Gamma_1$ such that $zP \prec_{M_\Gamma} M_{\Gamma \backslash w}$.
    \item[(b)]   $M_\Gamma$ is amenable.
\end{enumerate}

\vspace{0.3cm}

\noindent {\it Proof of the claim.}
Since $zA\prec_{M_\Gamma} M_{\Gamma_1\cap \Gamma_2}$ but $zA\not\prec_{M_\Gamma} \CC = M_{\emptyset}$ there is a non-empty subgraph $\Lambda \subseteq \Gamma_1 \cap \Gamma_2$ such that $zA \prec_{M_\Gamma} M_\Lambda$ but $zA \not \prec_{M_\Gamma} M_{\Lambda \backslash \{ w \}}$ for all $w \in \Lambda$.
There are two cases.

 First assume that $M_\Lambda$ is non-amenable, equivalently, by Theorem \ref{Thm=AmenableCoxeter}, $\Lambda$  contains $L$ or $L^+$. But then  $\Link(\Lambda)$  must be a clique as otherwise   $K_{2,3}$ or $K_{2,3}^+$ would be a subgraph of $\Gamma$.
Recall that we fixed $v \in \Gamma$ in the second paragraph of the induction part of this proof and we have set $\Gamma_1 = \Star(v)$ and $\Gamma_2 = \Gamma \backslash v$. Now as $\Lambda \subseteq \Gamma_1 \cap \Gamma_2 = \Link(v)$ we have $v \in \Link(\Lambda)$.
  Combining this with the fact that $\Link(\Lambda)$ is a clique it must follow that $\Link(\Lambda) \subseteq \Star(v) = \Gamma_1$.  As $\Gamma \backslash \Gamma_1$ was assumed to be non-empty   we may pick any vertex in $\Gamma \backslash \Gamma_1$ and we have proved (a) by Proposition \ref{Prop=LocateNormalizerNew} where $r = z$.

 Second, assume that $M_\Lambda$ is amenable. Note that $\Lambda$ must contain at least two points not connected by an edge as otherwise $M_\Lambda$ is finite dimensional and  $zA \prec_{M_\Gamma} M_{\Lambda}$ with $zA$ diffuse leads to a contradiction.
 If $\Lambda \cup \Link(\Lambda)$ is not equal to $\Gamma$ then any $w \in \Gamma \backslash (\Lambda \cup \Link(\Lambda))$ will yield (a) through Proposition \ref{Prop=LocateNormalizerNew}   as in the previous paragraph and the claim is proved.  So we assume $\Gamma  = \Lambda \cup \Link(\Lambda)$. But then   $\Gamma \backslash \Gamma_1$ contains at most one point,   since otherwise it contradicts that $K_{2,3}$ or $K_{2,3}^+$ is not a subgraph of $\Gamma$. Since we assumed  $\Gamma \backslash \Gamma_1$ is non-empty we see that $\Gamma \backslash \Gamma_1$  consists of exactly one point, say $x$. But then $\Gamma = \{ v, x \} \times \Link(v)$ and $M_\Gamma = M_{v,x} \otimes M_{\Link(v)}$ which is amenable and we are in case (b).

 \vspace{0.3cm}

 \noindent {\it Remainder of the proof of Case \eqref{Case1}.}
  In case (b) of the claim strong solidity is trivial. In case (a) of the claim it follows from Proposition \ref{Prop=EmbedSolid} that $zP$ with $z \not = 0$ contains an  amenable direct summand which is a contradiction. So $z=0$ and $P$ is amenable.  This concludes the proof in case \eqref{Case1}.

 \vspace{0.3cm}

 \noindent {\it Remainder of the proof of the main theorem in the situation that  Case \eqref{Case1} and Case \eqref{Case2} never occur.}
We first recall that if we can find a single vertex $v$ as above such that we are in case  \eqref{Case1} or \eqref{Case2} then the proof is finished. Otherwise for all vertices $v \in \Gamma$  we are in case  \eqref{Case3}.  So  $zP$ is amenable relative to $M_{\Link(v)}$ inside $M_\Gamma$. As $\bigcap_{v\in \Gamma}\Link(v)\subseteq \bigcap_{v\in \Gamma}\Gamma\setminus\{v\}=\emptyset$ we obtain  by using \cref{Thm=Square} repeatedly that $zP$ is amenable relative to $\CC$, i.e. $zP$ is amenable \cite[Proposition 2.4 (2)]{OzawaPopaAnnals}. So $z =0$ and we conclude again that $P$ is amenable.
\end{proof}

We now summarize our results.

\begin{theorem}\label{Thm=Dichotomy}
  Let $W = W_\Gamma$ be a right-angled Coxeter group. The following are equivalent:
  \begin{enumerate}
  \item\label{ItemSolid=1} $M_\Gamma = \cL(W_\Gamma)$ is not strongly solid.
  \item\label{ItemSolid=2} $W_\Gamma$ contains $\mathbb{Z} \times \mathbb{F}_2$ as a subgroup.
  \item\label{ItemSolid=3} $\Gamma$ contains $K_{2,3}$ or $K_{2,3}^+$ as a subgraph.
  \end{enumerate}
\end{theorem}
\begin{proof}  We first collect the easy implications. \eqref{ItemSolid=3} implies \eqref{ItemSolid=2} follows as the graph product $\ast_{K_{2,3}} \mathbb{Z}_2$ equals $(\mathbb{Z}_2 \ast \mathbb{Z}_2) \times (\mathbb{Z}_2 \ast \mathbb{Z}_2 \ast \mathbb{Z}_2) \simeq D_\infty \ast (D_\infty \ast \mathbb{Z}_2)$ and $(D_\infty \ast \mathbb{Z}_2)$ contains $\mathbb{F}_2$ as a subgroup. We have that  $\ast_{K_{2,3}^+} \mathbb{Z}_2 = (\mathbb{Z}_2 \ast \mathbb{Z}_2) \times ((\mathbb{Z}_2 \times \mathbb{Z}_2) \ast \mathbb{Z}_2)$ which contains again $\mathbb{Z} \times \mathbb{F}_2$.

\eqref{ItemSolid=2} implies \eqref{ItemSolid=1} follows as $\cL(\mathbb{Z} \times \mathbb{F}_2) \simeq \cL(\mathbb{Z}) \otimes \cL(\mathbb{F}_2)$. As  $\cL(\mathbb{Z}) \otimes 1$ is diffuse and amenable and its relative commutant contains  $\cL(\mathbb{Z}) \otimes \cL(\mathbb{F}_2)$ we see that its normalizer cannot generate an amenable von Neumann algebra. Therefore  $\cL(\mathbb{Z} \times \mathbb{F}_2)$ is not strongly solid and so neither  is $\cL(W)$ as strong solidity passes to von Neumann subalgebras.
 \eqref{ItemSolid=1} implies \eqref{ItemSolid=3} is proved in Theorem \ref{Thm=MainImplication}.
\end{proof}

\begin{remark}
    In particular Theorem \ref{Thm=Dichotomy} implies the following purely group theoretical result. If $W$ is a right-angled Coxeter group that contains $\mathbb{Z} \times \mathbb{F}_2$ as a subgroup, then actually   it contains   $D_\infty \times (\mathbb{Z}_2 \ast \mathbb{Z}_2 \ast \mathbb{Z}_2)$ or $D_\infty \times (\mathbb{Z}_2 \ast (\mathbb{Z}_2 \times \mathbb{Z}_2))$ as a special Coxeter subgroup.
\end{remark}

\printbibliography

\end{document}